\documentclass[11pt,a4paper]{article}
\usepackage{amsfonts}
\usepackage{amsmath,amssymb,amsthm,amsfonts,enumerate}
\allowdisplaybreaks[4]
\usepackage{graphicx}
\usepackage{float}
\usepackage[scriptsize,nooneline]{subfigure}
\usepackage{indentfirst}

\usepackage{cite}
\usepackage{color}
\usepackage{mathrsfs}
\usepackage{epstopdf}
\usepackage{setspace}
\usepackage[labelfont=footnotesize,font=footnotesize]{caption}
\usepackage[colorlinks, citecolor=red]{hyperref}
\usepackage{esint}

\newtheorem{theorem}{Theorem}[section]
\newtheorem{lemma}{Lemma}[section]
\newtheorem{definition}{Definition}[section]

\newtheorem{remark}{Remark}[section]

\def\geq{\geqslant}\def\leq{\leqslant}
\date{}
\begin{document}

\title{\bf Interpolation on Convex-set Valued Lebesgue Spaces and its Applications\\[10pt] \rm{\Large Yuxun Zhang, Jiang Zhou*}}
\maketitle
{\bf Abstract: \rm In this paper, we obtain two interpolation theorems on convex-set valued Lebesgue spaces, which generalize the Marcinkiewicz interpolation theorem and Riesz-Thorin interpolation theorem on classical Lebesgue spaces, respectively. As applications, we obtain the boundedness of convex-set valued fractional averaging operators and fractional maximal operators, and prove the reverse factorization property of matrix weights.



\baselineskip 15pt

\section{Introduction}

In 1972, the classical $A_p$ weights were first introduced by Muckenhoupt\cite{Mur1972} during the study of weighted norm inequalities for Hardy-Littlewood maximal operator. A weight (a non-negative, measurable function that satisfies $0<\omega(x)<\infty$, a.e.) $\omega$ is said to satisfy $\omega\in A_p$ ($1<p<\infty$) if
$$[\omega]_{A_p}:=\sup_Q\left(\frac{1}{m_n(Q)}\int_{Q}\omega(x)dx\right)\left(\frac{1}{m_n(Q)}\int_{Q}\omega(x)^{\frac{1}{1-p}}dx\right)^{p-1}<\infty,$$
where the supremum is taken over all cubes $Q\subset\mathbb{R}^n$ with sides parallel to the coordinate axes, $m_n(Q)$ denotes the n-dimensional Lebesgue measure of $Q$. A weight $\omega$ is said to satisfy $\omega\in A_1$ if for all $x\in\mathbb{R}^n$,
$$M\omega(x)\leq C\omega(x).$$

For $d\in\mathbb{N}$, given a Calder\'{o}n-Zygmund singular integral operator $T$, it can be extended to an operator on vector-valued functions $\vec{f}=(f_1,f_2,\cdots,f_d)^t$ by applying it to each coordinate: $T\vec{f}:=(Tf_1,Tf_2,\cdots,Tf_d)^t$.

Denote $\mathcal{M}_d$ to be the set of all $d\times d$ real matrices, a matrix function $W$ is a mapping $W:\mathbb{R}^n\rightarrow\mathcal{M}_d$. In a series of papers in the 1990s, Nazarov, Treil and Volberg\cite{Naz1996, Tre1997, Vol1997} considered that for $1<p<\infty$, what type of positive semidefinite, self-adjoint matrix functions $W$ make the weighted inequality
\begin{equation}
\int_{\mathbb{R}^n}|W^{\frac{1}{p}}(x)T(\vec{f})(x)|^pdx\leq C\int_{\mathbb{R}^n}|W^{\frac{1}{p}}(x)\vec{f}(x)|^pdx
\end{equation}
hold, and defined the matrix weight class $A_{p,matrix}$ to be the set of such matrix functions $W$, where $Uv\in\mathbb{R}^d$ denotes the product of $U\in\mathcal{M}_d$ and $v\in\mathbb{R}^d$, $|v|$ denotes the Euclidean norm of $v\in\mathbb{R}^d$.

In 2003, Roudenko\cite{Rou2003} gave a characterization of matrix  weights: $W\in A_{p,matrix}$ $(1<p<\infty)$ if and only if
\begin{equation}
[W]_{A_{p,matrix}}:=\sup_{Q}\left(\frac{1}{m_n(Q)}\int_Q\left(\frac{1}{m_n(Q)}\int_Q\lVert W^{\frac{1}{p}}(x)W^{-\frac{1}{p}}(y)\rVert_{\mathrm{op}}^{p'}dy\right)^{\frac{p}{p'}}dx\right)^{\frac{1}{p}}<\infty,
\end{equation}
where the supremum is taken over all cubes $Q\subset\mathbb{R}^n$ with sides parallel to the coordinate axes, $\lVert U\rVert_{\mathrm{op}}$ denotes the operator norm of the matrix $U\in\mathcal{M}_d$, that is,
$$\lVert U\rVert_{\mathrm{op}}=\max_{v\in\mathbb{R}^d}\frac{|Uv|}{|v|}.$$
\begin{remark}
The following discussion indicates that matrix $A_{p,matrix}$ weights are the generalization of classical $A_p$ weights: while $d=1$, the matrix function $W$ will be a non-negative real valued function $\omega$, then the inequality (1) will turn into
$$\int_{\mathbb{R}^n}|T(f)(x)|^p\omega(x)dx\leq C\int_{\mathbb{R}^n}|f(x)|^p\omega(x)dx,$$
and the condition (2) will turn into
$$\sup_Q\left(\frac{1}{m_n(Q)}\int_Q\omega(x)dx\right)^{\frac{1}{p}}\left(\frac{1}{m_n(Q)}\int_Q\omega(x)^{-\frac{p'}{p}}dx\right)^{\frac{1}{p'}}<\infty,$$
which are the weighted norm inequality and the classical $A_p$ condition, respectively.
\end{remark}

Jones factorization theorem\cite{Jon1980} and Rubio de Francia extrapolation\cite{Rub1991} are two important theorems in classical $A_p$ weight theory. In 1996, Nazarov and Treil\cite{Naz1996} conjectured that these results can be generalized to matrix weights, which generated two longstanding open problems. In 2022, Bownik and Cruz-uribe\cite{Mar2022} solved these problems by creatively utilizing convex analysis theory (see, for instance, \cite{Aub2009, Cas1977, Roc1970}). They defined some new spaces consisted of convex-set valued functions, including $L_{\mathcal{K}}^p(\Omega,\rho)$ and their special form $L_{\mathcal{K}}^p(\Omega,|\cdot|)$ $(1\leq p\leq\infty)$, and furthermore obtained that for $1<p\leq\infty$, a norm function $\rho$ satisfying $\mathcal{A}_{p,norm}$ condition (see\cite[Definition 6.2]{Mar2022}) ensures the boundedness of convex-set valued maximal operator on $L_{\mathcal{K}}^p(\Omega,\rho)$, which enables the Rubio de Francia iteration algorithm to be used and this conjecture to be ultimately proved.

Interpolation is a useful method to obtain the boundedness of several sublinear operators. In 1939, Marcinkiewicz\cite{Mar1939} first put forward the so-called Marcinkiewicz interpolation theorem, and Zygmund\cite{Zyg1956} reintroduced it in 1956. In 1982, in order to study the commutators generated by fractional integrals and BMO functions, Chanillo \cite{Cha1982} obtained the boundedness of fractional maximal operators first introduced in\cite{Mur1974} by using a generalization of Marcinkiewicz interpolation theorem in\cite{Zyg1959}.

Riesz\cite{Rie1927} first proved the Riesz convexity theorem in 1927, which is the original version of Riesz-Thorin interpolation theorem. After that, Thorin \cite{Tho1938, Tho1948} generalized this result, and published his theses with a variety of applications. Riesz-Thorin interpolation theorem requires strong boundedness of operators at endpoints, but has the better constant and lower requirment of exponents, and is easier to be generalized. For example, in 1958, Stein and Weiss obtained the Stein-Weiss variable measure interpolation theorem\cite{Ding2008}, which can be used to obtain the reverse factorization property of $A_p$ weights.

To be specific, given $1<p_0,p_1<\infty$, suppose that $w_0\in A_{p_0}$ and $w_1\in A_{p_1}$, the classical $A_p$ weight theory\cite{Mur1972} indicates that


$$\lVert Mf\rVert_{L^{p_0}_{w_0}}\leq C_0\lVert f\rVert_{L^{p_0}_{w_0}},\ \ \lVert Mf\rVert_{L^{p_1}_{w_1}}\leq C_1\lVert f\rVert_{L^{p_1}_{w_1}},$$
where $M$ is the Hardy-Littlewood operator, and for $r>0$,
$$\lVert g\rVert_{L^r_{w}}=\left(\int_{\mathbb{R}^n}|g(x)|^rw(x)dx\right)^{\frac{1}{r}}.$$
Then by Stein-Weiss variable measure interpolation theorem\cite{Ding2008}, for
$$\frac{1}{p}=\frac{1-t}{p_0}+\frac{t}{p_1},\ \ w=w_0^{\frac{p(1-t)}{p_0}}w_1^{\frac{pt}{p_1}},\ \ C=C_0^{1-t}C_1^t,$$
there holds $\lVert Mf\rVert_{L^{p}_{w}}\leq C\lVert f\rVert_{L^{p}_{w}}$, which implies that $w\in A_p$.

In this paper, we will further investigate the properties of convex-set valued Lebesgue spaces, especially establish the interpolation theory above them, including Marcinkiewicz interpolation theorem and Riesz-Thorin interpolation theorem. These results not only enrich and develop set-valued analysis theory, but also have some applications in harmonic analysis.

The remainder of this paper is organized as follows. In Section 2, we give some necessary definitions and lemmas. In Section 3, we prove the Marcinkiewicz interpolation theorem, define the fractional averaging operator and fractional maximal operator on convex-set valued Lebesgue spaces, and obtain their boundedness. In Section 4, we prove the Riesz-Thorin interpolation theorem, and use it with some known results of norm functions to obtain the reverse factorization property of matrix weights, which is a part of Jones factorization theorem first proved in\cite{Mar2022}.

Throughout this paper we will use the following notations. We will use $\mathbb{R}^n$ to denote Euclidean space of dimension $n$, which will be the domain of functions, and use $d$ to denote the dimension of vector and set-valued functions. We will use $(\Omega,\mathcal{A},\mu)$ to denote a positive, $\sigma$-infinite, and complete measurable spaces, where $\Omega\subset\mathbb{R}^n$. The open unit ball $\{v\in\mathbb{R}^d:|v|<1\}$ will be denoted by $\boldsymbol{B}$ and its closure by $\boldsymbol{\overline{B}}$. The set of all $d\times d$, real symmetric, positive semidefinite matrices will be denoted by $\mathcal{S}_d$. For $1\leq p\leq\infty$, $p'$ will denote the conjugate index of $p$ such that $1/p+1/p'=1$, $L^p(\Omega)$ will denote the Lebesgue space of real-valued functions on $\Omega$, and $L^p(\Omega,\mathbb{R}^d)$ will denote the Lebesgue space of vector-valued functions on $\Omega$. Given two quantities $A$ and $B$, we will write $A\lesssim B$ if there exists a constant $C>0$ such that $A\leq CB$, and write $A\approx B$ if $A\lesssim B$ and $B\lesssim A$.

\section{Preliminary}
We begin with some basic definitions.
\begin{definition}\cite{Roc1970}
For $E,F\subset\mathbb{R}^d$, define $\overline{E}$ to be the closure of $E$, $E+F=\{x+y:x\in E,y\in F\}$. For $\lambda\in\mathbb{R}$, define $\lambda E=\{\lambda x:x\in E\}$.
\end{definition}

\begin{definition}\cite{Roc1970}
A set $E\subset\mathbb{R}^d$ is bounded if $E\subset B$ for some ball $B\subset\mathbb{R}^d$; $E$ is convex if for all $x,y\in E$ and $0<\lambda<1$, $\lambda x+(1-\lambda)y\in E$; $E$ is symmetric if $-E=E$.
\end{definition}

Let $\mathcal{K}(\mathbb{R}^d)$ be the collection of all closed, nonempty subsets of $\mathbb{R}^d$, and the subscripts $b,c,s$ appended to $\mathcal{K}$ will denote bounded, convex and symmetric sets, respectively. For example, $\mathcal{K}_{bcs}(\mathbb{R}^d)$ denotes bounded, convex, symmetric and closed subsets of $\mathbb{R}^d$.

\begin{definition}\cite{Roc1970}
Given a set $E\subset\mathbb{R}^d$, let $\mathrm{conv}(E)$ denote the convex hull of $E$:
$$\mathrm{conv}(E)=\left\{\sum_{i=1}^{k}\alpha_ix_i:x_i\in E, \alpha_i\geq0, \sum_{i=1}^{k}\alpha_i=1\right\},$$
and denote the closure of the convex hull of $E$ by $\overline{\mathrm{conv}}(E)$.
\end{definition}

\begin{definition}\cite{Roc1970}
A seminorm is a function $p:\mathbb{R}^d\rightarrow[0,\infty)$ that satisfies the following properties: for all $u,v\in\mathbb{R}^d$ and $\alpha\in\mathbb{R}$,
$$p(u+v)\leq p(u)+p(v),\ \ p(\alpha v)=|\alpha|p(v).$$
For a seminorm $p$ and $E\subset\mathbb{R}^d$, define
$$p(E)=\sup_{v\in E}p(v).$$
A seminorm is a norm if $p(v)=0$ equals $v=0$.
\end{definition}

\begin{definition}\cite{Mar2022}
A seminorm function $\rho$ on $\Omega$ is a mapping $\rho:\Omega\times\mathbb{R}^d\rightarrow[0,\infty)$ such that:

(i) for all $x\in\Omega$, $\rho_x(\cdot)=\rho(x,\cdot)$ is a seminorm on $\mathbb{R}^d$;

(ii) $x\mapsto\rho_x(v)$ is a measurable function for any $v\in\mathbb{R}^d$.
\end{definition}

In what follows, to simplify writing, we abbreviate $[\rho_x(E)]^p$ as $\rho_x(E)^p$. Next, definitions of measurable set-valued functions and their integrals will be introduced.
\begin{definition}\cite{Aub2009}
Given a function $F:\Omega\rightarrow\mathcal{K}(\mathbb{R}^d)$, $F$ is measurable if for every open set $U\subset\mathbb{R}^d$, $F^{-1}(U):=\{x\in\Omega:F(x)\cap U\neq\phi\}\in\mathcal{A}$.
\end{definition}

\begin{definition}\cite{Aub2009}
Given a function $F:\Omega\rightarrow\mathcal{K}(\mathbb{R}^d)$, denote $S^0(\Omega,F)$ to be the set of all measurable selection functions of $F$, that is, all measurable functions $f$ such that $f(x)\in F(x),a.e.$
\end{definition}

By\cite[Theorem 8.1.4]{Aub2009}, for a measurable function $F$, $S^0(\Omega,F)$ is not empty.

\begin{definition}\cite{Aub2009}
Given a measurable function $F:\Omega\rightarrow\mathcal{K}(\mathbb{R}^d)$, define the set of all integrable selection functions of $F$ by
$$S^1(\Omega,F):=S^0(\Omega,F)\cap L^1(\Omega,\mathbb{R}^d).$$
The Aumann integral of $F$ is the set of integrals of integrable selection functions of $F$, i.e.
$$\int_{\Omega}Fd\mu:=\left\{\int_{\Omega}fd\mu:f\in S^1(\Omega,F)\right\}.$$
\end{definition}

There may not exist integrable selection function of $F$, but the integrably bounded property ensures the existence.

\begin{definition}\cite{Aub2009}
A measurable function $F:\Omega\rightarrow\mathcal{K}(\mathbb{R}^d)$ is integrably bounded if there exists a non-negative function $k\in L^1(\Omega,\mathbb{R})$ such that
$$F(x)\subset B(0,k(x))\ \ for\ a.e.\ x\in\Omega.$$
If $\Omega$ is a metric space, say $F$ is locally integrably bounded if this holds for $k\in L_{loc}^1(\Omega,\mathbb{R})$.
\end{definition}

\begin{remark}
If $f\in L^1(\Omega,\mathbb{R}^d)$ and $F(x)=\overline{\mathrm{conv}}(\{f(x),-f(x)\})$ for $x\in\Omega$, then $F$ is integrably bounded.
\end{remark}

Similar to Lebesgue integral of real-valued functions, Aumann integral is linear and monotone.

\begin{lemma}\cite{Aub2009}
Suppose that $F_1,F_2:\Omega\rightarrow\mathcal{K}_{bcs}(\mathbb{R}^d)$ are measurable and integrably bounded. Then for any $\lambda_1,\lambda_2\in\mathbb{R}$,
$$\int_{\Omega}(\lambda_1F_1+\lambda_2F_2)d\mu=\alpha_1\int_{\Omega}F_1d\mu+\alpha_2\int_{\Omega}F_2d\mu.$$
If $F_1(x)\subset F_2(x)$ for any $x\in\Omega$, then
$$\int_{\Omega}F_1d\mu\subset\int_{\Omega}F_2d\mu.$$
\end{lemma}

The following lemma can be seen in the proof of\cite[Proposition 5.2]{Mar2022}.

\begin{lemma}
Suppose that $F:\Omega\rightarrow\mathcal{K}_{bcs}(\mathbb{R}^d)$ is measurable and integrably bounded, $E\subset\Omega$, then
$$\left|\int_EFd\mu\right|\leq\int_E|F|d\mu.$$
\end{lemma}

This chapter ends with definitions of convex-set valued Lebesgue spaces ${L_{\mathcal{K}}^p(\Omega,\rho)}$,  convex-set valued weak Lebesgue spaces ${L_{\mathcal{K}}^{p,\infty}(\Omega,\rho)}$ and distribution functions.

\begin{definition}\cite{Mar2022}
Suppose that $\rho$ is a seminorm function on $\Omega$. For any $p\in(0,\infty)$, define $L_{\mathcal{K}}^p(\Omega,\rho)$ to be the set of all measurable functions $F:\Omega\rightarrow\mathcal{K}_{bcs}(\mathbb{R}^d)$ such that
$$\lVert F\rVert_{L_{\mathcal{K}}^p(\Omega,\rho)}:=\left(\int_{\Omega}\rho_x(F(x))^pd\mu(x)\right)^{\frac{1}{p}}<\infty,$$
and $L_{\mathcal{K}}^{p,\infty}(\Omega,\rho)$ to be the set of all such $F:\Omega\rightarrow\mathcal{K}_{bcs}(\mathbb{R}^d)$ that satisfy
$$\lVert F\rVert_{L_{\mathcal{K}}^{p,\infty}(\Omega,\rho)}:=\sup_{\lambda>0}\lambda\mu(\{x\in\Omega:\rho_x(F(x))>\lambda\})^{\frac{1}{p}}<\infty.$$
When $p=\infty$, define $L_{\mathcal{K}}^{\infty}(\Omega,\rho)$ to be the set of all such $F$ that satisfy
$$\lVert F\rVert_{L_{\mathcal{K}}^{\infty}(\Omega,\rho)}:=\mathop{\mathrm{ess}\sup}\limits_{x\in\Omega}\rho_x(F(x))<\infty.$$
\end{definition}

\begin{remark}
If $\rho_x(v)=|v|$ for any $(x,v)\in\Omega\times\mathbb{R}^d$, the space $L_{\mathcal{K}}^p(\Omega,\rho)$ will be $L_{\mathcal{K}}^p(\Omega,|\cdot|)$,
which is the generalization of $L^p(\Omega,\mathbb{R}^d)$. In fact, given $0<p\leq\infty$ and $f\in L^p(\Omega,\mathbb{R}^d)$, set $F(x)=\overline{\mathrm{conv}}(\{f(x),-f(x)\})$ for $x\in\Omega$, then $F\in L_{\mathcal{K}}^p(\Omega,|\cdot|)$ and $\lVert F\rVert_{L_{\mathcal{K}}^p(\Omega,|\cdot|)}=\lVert f\rVert_{L^p(\Omega,\mathbb{R}^d)}$.
\end{remark}

\begin{definition}
Suppose that $F:\Omega\rightarrow\mathcal{K}_{bcs}(\mathbb{R}^d)$ is measurable, $\rho$ is a seminorm function on $\Omega$. For any $p\in(0,\infty)$, define the distribution function of $F$ as
$$\omega_F(\lambda)=\mu(\{x\in\Omega:\rho_x(F(x))>\lambda\})\ \ for\ \lambda>0.$$
\end{definition}

Since $\rho_x(F(x))$ is a real-valued function in $L^p(\Omega)$, it's obvious that the following integral formula holds, we omit the proof.

\begin{lemma}
For any $p\in(0,\infty)$ and $F\in L_{\mathcal{K}}^p(\Omega,\rho)$,
$$\lVert F\rVert_{L_{\mathcal{K}}^p(\Omega,\rho)}^p=p\int_{0}^{\infty}\lambda^{p-1}\omega_F(\lambda)d\lambda.$$
\end{lemma}

\section{Marcinkiewicz interpolation theorem and its application}
We first prove Marcinkiewicz interpolation theorem.

\begin{theorem}
Let $1\leq p_0\leq q_0\leq\infty$, $1\leq p_1\leq q_1\leq\infty$, $q_0\neq q_1$, $\rho$ be a seminorm function on $\Omega$, $T$ be a sublinear operator defined on the space $L^{p_0}_\mathcal{K}(\Omega,\rho)+L^{p_1}_\mathcal{K}(\Omega,\rho)$: for all $F,G\in L^{p_0}_\mathcal{K}(\Omega,\rho)+L^{p_1}_\mathcal{K}(\Omega,\rho)$, $\lambda\in\mathbb{R}$ and $x\in\Omega$,
$$T(F+G)(x)\subset TF(x)+TG(x),\ \ T(\lambda F)(x)=\lambda TF(x).$$
Suppose that there exist two positive constants $C_0$ and $C_1$ such that
$$\lVert T(F)\rVert_{L^{q_0,\infty}_\mathcal{K}(\Omega,\rho)}\leq C_0\lVert f\rVert_{L^{p_0}_\mathcal{K}(\Omega,\rho)}\ \ for\ all\ F\in L^{p_0}_\mathcal{K}(\Omega,\rho),$$
$$\lVert T(F)\rVert_{L^{q_1,\infty}_\mathcal{K}(\Omega,\rho)}\leq C_1\lVert f\rVert_{L^{p_1}_\mathcal{K}(\Omega,\rho)}\ \ for\ all\ F\in L^{p_1}_\mathcal{K}(\Omega,\rho),$$
then for all $0<t<1$ and $F\in L^{p}_\mathcal{K}(\Omega,\rho)$, we have
\begin{equation}
\lVert T(F)\rVert_{L^{q}_\mathcal{K}(\Omega,\rho)}\leq C_t\lVert F\rVert_{L^{p}_\mathcal{K}(\Omega,\rho)},
\end{equation}
where
$$\frac{1}{p}=\frac{1-t}{p_0}+\frac{t}{p_1},\ \ \frac{1}{q}=\frac{1-t}{q_0}+\frac{t}{q_1},\ \ and\ \ C_t=2\left(\frac{q}{|q-q_0|}+\frac{q}{|q-q_1|}\right)^{\frac{1}{q}}C_0^{1-t}C_1^t.$$
\end{theorem}

\begin{proof}
Without the loss of generality, assume that $p_1\leq p_0$, and consider the following cases separately:
$$(a_1)\ p_1<p_0,\ q_1<q_0<\infty;\ \ (a_2)\ p_1<p_0,\ q_0<q_1<\infty;$$
$$(b)\ p_1=p_0,\ q_1<q_0<\infty;\ \ (c_1)\ p_1<p_0,\ q_1<q_0=\infty;$$
$$(c_2)\ p_1<p_0,\ q_0<q_1=\infty;\ \ (d)\ p_1=p_0,\ q_1<q_0=\infty.$$

\textbf{Case} $\boldsymbol{(a_1).}$ For a fixed $z>0$ and any $F\in L^{p}_\mathcal{K}(\Omega,\rho)$, split $F=F_0+F_1$, where
$$F_0(x)=\left\{\begin{array}{rcl}
F(x),\ \ \ \ \rho(F(x))\leq z,\\
\{0\},\ \ \ \ \rho(F(x))>z,\end{array}
\right.\ \ F_1(x)=\left\{\begin{array}{rcl}
\{0\},\ \ \ \ \rho(F(x))\leq z,\\
F(x),\ \ \ \ \rho(F(x))>z.\end{array}
\right.$$
By the definition of $F_0$ and $F_1$, we have
$$\int_{\Omega}\rho(F_0(x))^{p_0}d\mu(x)=\int_{\Omega}\rho(F_0(x))^{p_0-p}\rho(F_0(x))^pd\mu(x)\leq z^{p_0-p}\int_{\Omega}\rho(F(x))^pd\mu(x),$$
$$\int_{\Omega}\rho(F_1(x))^{p_1}d\mu(x)=\int_{\Omega}\rho(F_1(x))^{p_1-p}\rho(F_1(x))^pd\mu(x)\leq z^{p_1-p}\int_{\Omega}\rho(F(x))^pd\mu(x),$$
thus $F_0\in L^{p_0}_\mathcal{K}(\Omega,\rho)$ and $F_1\in L^{p_1}_\mathcal{K}(\Omega,\rho)$.\\
By the sublinearity property of $T$ and $\rho$, there holds
$$\rho(T(F)(x))\leq\rho(T(F_0)(x))+\rho(T(F_1)(x)),$$
which implies that, for any $\lambda>0$,
$$\left\{x:\rho(T(F)(x))>2\lambda\right\}\subset\left\{x:\rho(T(F_0)(x))>\lambda\right\}\cup\left\{x:\rho(T(F_1)(x))>\lambda\right\},$$
therefore
$$\omega_{T(F)}(2\lambda)\leq\omega_{T(F_0)}(\lambda)+\omega_{T(F_1)}(\lambda).$$
The weak boundedness of $T$ gives
\begin{align*}
\omega_{T(F)}(2\lambda)&\leq\frac{C_0^{q_0}}{\lambda^{q_0}}\left(\int_{\{x\in\Omega:\rho(F(x))\leq z\}}\rho(F(x))^{p_0}d\mu(x)\right)^{k_0}+\frac{C_1^{q_1}}{\lambda^{q_1}}\left(\int_{\{x\in\Omega:\rho(F(x))>z\}}\rho(F(x))^{p_1}d\mu(x)\right)^{k_1},
\end{align*}
where $k_0=q_0/p_0$, $k_1=q_1/p_1$ are not less than $1$.
Then, by Lemma 2.3, we obtain that
\begin{align*}
\lVert T(f)\rVert_{L_{\mathcal{K}}^q(\Omega,\rho)}^q&=q\int_0^{\infty}\lambda^{q-1}\omega_{T(F)}(\lambda)d\lambda=2^qq\int_0^{\infty}\lambda^{q-1}\omega_{T(F)}(2\lambda)d\lambda\\
&\leq 2^qqC_0^{q_0}\int_{0}^{\infty}\lambda^{q-q_0-1}\left(\int_{\rho(F(x))\leq z}\rho(F(x))^{p_0}d\mu(x)\right)^{k_0}d\lambda\\
&\ \ \ \ +2^qqC_1^{q_1}\int_{0}^{\infty}\lambda^{q-q_1-1}\left(\int_{\rho(F(x))>z}\rho(F(x))^{p_1}d\mu(x)\right)^{k_1}d\lambda\\
&=:2^qqC_0^{q_0}I_0+2^qqC_1^{q_1}I_1.
\end{align*}
By the dual property of weighted Lebesgue spaces, we have
$$I_0^{\frac{1}{k_0}}=\sup_{g_0}\int_{0}^{\infty}\lambda^{q-q_0-1}\left(\int_{\rho(F(x))\leq z}\rho(F(x))^{p_0}d\mu(x)\right)g_0(\lambda)d\lambda,$$
$$I_1^{\frac{1}{k_1}}=\sup_{g_1}\int_{0}^{\infty}\lambda^{q-q_1-1}\left(\int_{\rho(F(x))>z}\rho(F(x))^{p_1}d\mu(x)\right)g_1(\lambda)d\lambda,$$
where $g_0$ and $g_1$ are real-valued functions satisfying
$$\int_{0}^{\infty}\lambda^{q-q_0-1}|g_0(\lambda)|^{k'_0}d\lambda\leq 1,\ \ \int_{0}^{\infty}\lambda^{q-q_1-1}|g_1(\lambda)|^{k'_1}d\lambda\leq 1.$$
Set $z=(\lambda/A)^{\xi}$, where $A$ and $\xi$ will be determined later. By Fubini's theorem, the integral under the first 'sup' equals
\begin{align*}
&\int_{\Omega}\rho(F(x))^{p_0}\left(\int_{A\rho(F(x))^{\frac{1}{\xi}}}^{\infty}\lambda^{q-q_0-1}g_0(\lambda)d\lambda\right)d\mu(x)\\
&\ \ \ \ \leq\int_{\Omega}\rho(F(x))^{p_0}\left(\int_{A\rho(F(x))^{\frac{1}{\xi}}}^{\infty}\lambda^{q-q_0-1}d\lambda\right)^{\frac{1}{k_0}}\left(\int_{A\rho(F(x))^{\frac{1}{\xi}}}^{\infty}\lambda^{q-q_0-1}g_0(\lambda)^{\frac{1}{k'_0}}d\lambda\right)^{\frac{1}{k'_0}}d\mu(x)\\
&\ \ \ \ \leq\frac{1}{(q_0-q)^{\frac{1}{k_0}}}A^{\frac{q-q_0}{k_0}}\int_{\Omega}\rho(F(x))^{p_0+\frac{q-q_0}{k_0\xi}}d\mu(x),
\end{align*}
and the integral under the second 'sup' equals
\begin{align*}
&\int_{\Omega}\rho(F(x))^{p_1}\left(\int_{0}^{A\rho(F(x))^{\frac{1}{\xi}}}\lambda^{q-q_1-1}g_1(\lambda)d\lambda\right)d\mu(x)\\
&\ \ \ \ \leq\int_{\Omega}\rho(F(x))^{p_1}\left(\int_{0}^{A\rho(F(x))^{\frac{1}{\xi}}}\lambda^{q-q_1-1}d\lambda\right)^{\frac{1}{k_1}}\left(\int_{0}^{A\rho(F(x))^{\frac{1}{\xi}}}\lambda^{q-q_1-1}g_1(\lambda)^{\frac{1}{k'_1}}d\lambda\right)^{\frac{1}{k'_1}}d\mu(x)\\
&\ \ \ \ \leq\frac{1}{(q-q_1)^{\frac{1}{k_1}}}A^{\frac{q-q_1}{k_1}}\int_{\Omega}\rho(F(x))^{p_1+\frac{q-q_1}{k_1\xi}}d\mu(x).
\end{align*}
Note that
$$\frac{q-q_0}{k_0(p-p_0)}=\frac{\frac{1}{p}\left(\frac{1}{q_0}-\frac{1}{q}\right)}{\frac{1}{q}\left(\frac{1}{p_0}-\frac{1}{p}\right)}=\frac{\frac{t}{p}\left(\frac{1}{q_0}-\frac{1}{q_1}\right)}{\frac{t}{q}\left(\frac{1}{p_0}-\frac{1}{p_1}\right)},\ \ \frac{q-q_1}{k_1(p-p_1)}=\frac{\frac{1}{p}\left(\frac{1}{q_1}-\frac{1}{q}\right)}{\frac{1}{q}\left(\frac{1}{p_1}-\frac{1}{p}\right)}=\frac{\frac{1-t}{p}\left(\frac{1}{q_1}-\frac{1}{q_0}\right)}{\frac{1-t}{q}\left(\frac{1}{p_1}-\frac{1}{p_0}\right)},$$
we select
$$\xi=\frac{q-q_0}{k_0(p-p_0)}=\frac{q-q_1}{k_1(p-p_1)},$$
which makes
$$p_0+\frac{q-q_0}{k_0\xi}=p_1+\frac{q-q_1}{k_1\xi}=p,$$
therefore,
$$\lVert T(F)\rVert_{L_{\mathcal{K}}^q(\Omega,\rho)}^q\leq 2^qqC_0^{q_0}\frac{A^{q-q_0}}{q_0-q}\lVert F\rVert_{L_{\mathcal{K}}^p(\Omega,\rho)}^{pk_0}+ 2^qqC_1^{q_1}\frac{A^{q-q_1}}{q-q_1}\lVert F\rVert_{L_{\mathcal{K}}^p(\Omega,\rho)}^{pk_1}.$$
Next, select
$$A=C_0^{\frac{q_0}{q_0-q_1}}C_1^{\frac{q_1}{q_1-q_0}}\lVert F\rVert_{L_{\mathcal{K}}^p(\Omega,\rho)}^{\frac{p(k_0-k_1)}{q_0-q_1}},$$
therefore,
$$\lVert T(F)\rVert_{L_{\mathcal{K}}^q(\Omega,\rho)}^q\leq 2^qC_0^{q(1-t)}C_1^{qt}\left(\frac{q}{q_0-q}+\frac{q}{q-q_1}\right)\lVert F\rVert_{L_{\mathcal{K}}^p(\Omega,\rho)}^q,$$
thus we obtain (3).

\textbf{Case} $\boldsymbol{(a_2).}$ The proof of this case is similar to that in case $(a_1)$. We change the integration regions $(A\rho(F(x))^{1/\xi},\infty)$, $(0,A\rho(F(x))^{1/\xi})$ to $(0,A\rho(F(x))^{1/\xi})$, $(A\rho(F(x))^{1/\xi},\infty)$ respectively in the corresponding integers, and change the denominator $(q_0-q)^{1/k_0}$, $(q-q_1)^{1/k_1}$ to $(q-q_0)^{1/k_0}$, $(q_1-q)^{1/k_1}$. Then,
$$\lVert T(F)\rVert_{L_{\mathcal{K}}^q(\Omega,\rho)}^q\leq 2^qC_0^{q(1-t)}C_1^{qt}\left(\frac{q}{q-q_0}+\frac{q}{q_1-q}\right)\lVert F\rVert_{L_{\mathcal{K}}^p(\Omega,\rho)}^q,$$
which also implies (3).

\textbf{Case} $\boldsymbol{(b).}$ Note that $p=p_1=p_0$, by the weak boundedness of $T$, there hold
$$\omega_{T(F)}(\lambda)\leq\left(\frac{C_0\lVert F\rVert_{L_{\mathcal{K}}^p(\Omega,\rho)}}{\lambda}\right)^{q_0},\ \ \omega_{T(F)}(\lambda)\leq\left(\frac{C_1\lVert F\rVert_{L_{\mathcal{K}}^p(\Omega,\rho)}}{\lambda}\right)^{q_1}.$$
Therefore, for a certain $A$ which will be determined later, we have
\begin{align*}
\lVert T(f)\rVert_{L_{\mathcal{K}}^q(\Omega,\rho)}^q&=q\int_0^{\infty}\lambda^{q-1}\omega_{T(F)}(\lambda)d\lambda\\
&=q\left(\int_A^{\infty}\lambda^{q-1}\omega_{T(F)}(\lambda)d\lambda+\int_0^{A}\lambda^{q-1}\omega_{T(F)}(\lambda)d\lambda\right)\\
&\leq qC_0^{q_0}\lVert F\rVert_{L_{\mathcal{K}}^p(\Omega,\rho)}^{q_0}\int_{A}^{\infty}\lambda^{q-q_0-1}d\lambda+qC_1^{q_1}\lVert F\rVert_{L_{\mathcal{K}}^p(\Omega,\rho)}^{q_1}\int_{0}^{A}\lambda^{q-q_1-1}d\lambda\\
&=qC_0^{q_0}\frac{A^{q-q_0}}{q_0-q}\lVert F\rVert_{L_{\mathcal{K}}^p(\Omega,\rho)}^{q_0}+ qC_1^{q_1}\frac{A^{q-q_1}}{q-q_1}\lVert F\rVert_{L_{\mathcal{K}}^p(\Omega,\rho)}^{q_1}.
\end{align*}
Select
$$A=C_0^{\frac{q_0}{q_0-q_1}}C_1^{\frac{q_1}{q_1-q_0}}\lVert F\rVert_{L_{\mathcal{K}}^p(\Omega,\rho)},$$
therefore,
$$\lVert T(f)\rVert_{L_{\mathcal{K}}^q(\Omega,\rho)}^q\leq C_0^{q(1-t)}C_1^{qt}\left(\frac{q}{q_0-q}+\frac{q}{q-q_1}\right)\lVert F\rVert_{L_{\mathcal{K}}^p(\Omega,\rho)}^q,$$
which implies (3) with a better constant.

\textbf{Case} $\boldsymbol{(c_1).}$ If $p_0=\infty$, take $z=\delta/C_0$ in the split of $F$ in case $(a_1)$, then
$$\mathrm{ess}\sup\rho(T(F_0))\leq C_0\mathrm{ess}\sup\rho(F_0)\leq\delta,$$
which follows that $\omega_{T(F_0)}(\delta)=0$. Note that $A=C_0$, according to the similar calculation as which in case $(a_1)$, we have
$$\lVert T(F)\rVert_{L_{\mathcal{K}}^q(\Omega,\rho)}^q\leq 2^qC_0^{q(1-t)}C_1^{qt}\frac{q}{q-q_1}\lVert F\rVert_{L_{\mathcal{K}}^p(\Omega,\rho)}^{q}.$$
If $p_0<\infty$, we still select $z$ makes $\omega_{T(F_0)}(\delta)=0$, after which the proof proceeds as before.
Suppose that
$$z=\left(\frac{\delta}{A}\right)^{\xi},\ \ \textrm{where}\ \ \xi=\frac{p_0}{p_0-p},\ \ A=\eta C_0\lVert F\rVert_{L_{\mathcal{K}}^p(\Omega,\rho)}^{\frac{p}{p_0}},$$
and $\eta$ is a positive number which will be determined later. It is enough to show that for a suitable $\eta$, we have $\mathrm{ess}\sup\rho(T(F_0))\leq\delta$.
By the weak boundedness of $T$,
$$\mathrm{ess}\sup\rho(T(F_0))\leq C_0\lVert F\rVert_{L_{\mathcal{K}}^{p_0}(\Omega,\rho)}=C_0\left(p_0\int_{0}^{z}\lambda^{p_0-1}\omega_{F}(\lambda)d\lambda\right)^{\frac{1}{p_0}}.$$
It follows that we certainly have $\mathrm{ess}\sup\rho(T(F_0))\leq\delta$ if
$$C_0^{p_0}p_0\int_{0}^{z}\lambda^{p_0-1}\omega_{F}(\lambda)d\lambda\leq(Az^{\frac{1}{\xi}})^{p_0}=A^{p_0}z^{p_0-p},$$
which can be implied by
$$C_0^{p_0}p_0z^{p_0-p}\int_{0}^{\infty}\lambda^{p-1}\omega_{F}(\lambda)d\lambda\leq\eta^{p_0}C_0^{p_0}\lVert F\rVert_{L_{\mathcal{K}}^p(\Omega,\rho)}^pz^{p_0-p}.$$
Since the integral on the left is $p^{-1}\lVert F\rVert_{L_{\mathcal{K}}^p(\Omega,\rho)}^p$, the inequality is satisfied if $\eta^{p_1}\geq p_1/p$, which completes the proof of this case.

\textbf{Case} $\boldsymbol{(c_2).}$ In this case, we need only change the symbol of some exponents in case $(c_1)$ as which in case $(a_2)$. We omit the details here.

\textbf{Case} $\boldsymbol{(d).}$ By the weak boundedness of $T$, there hold
$$\mathrm{ess}\sup\rho(T(F))\leq C_0\lVert F\rVert_{L_{\mathcal{K}}^p(\Omega,\rho)},\ \ \omega_{T(F)}(\lambda)\leq\left(\frac{C_1\lVert F\rVert_{L_{\mathcal{K}}^p(\Omega,\rho)}}{\lambda}\right)^{q_1}.$$
Therefore, for $A=C_0\lVert F\rVert_{L_{\mathcal{K}}^p(\Omega,\rho)}$, we have
\begin{align*}
\lVert T(f)\rVert_{L_{\mathcal{K}}^q(\Omega,\rho)}^q&=q\int_0^{\infty}\lambda^{q-1}\omega_{T(F)}(\lambda)d\lambda=q\int_0^{A}\lambda^{q-1}\omega_{T(F)}(\lambda)d\lambda\\
&\leq qC_1^{q_1}\lVert F\rVert_{L_{\mathcal{K}}^p(\Omega,\rho)}^{q_1}\int_{0}^{A}\lambda^{q-q_1-1}d\lambda=qC_1^{q_1}\frac{A^{q-q_1}}{q-q_1}\lVert F\rVert_{L_{\mathcal{K}}^p(\Omega,\rho)}^{q_1}\\
&=C_0^{q(1-t)}C_1^{qt}\frac{q}{q-q_1}\lVert F\rVert_{L_{\mathcal{K}}^p(\Omega,\rho)}^q,
\end{align*}
which implies (3) with a better constant.
\end{proof}

The following discussion shows that Theorem 3.1 can deduce Marcinkiewicz interpolation theorem on Lebesgue spaces with an additional condition of operator. Through the same method, we can obtain Marcinkiewicz interpolation theorem on vector-valued Lebesgue spaces by Theorem 3.1.
\begin{remark}
Under the assumptions of $p_0,p_1,q_0,q_1$ in Theorem 3.1, suppose that $T_0$ is a sublinear operator bounded from $L^{p_0}(\Omega)$ to $L^{q_0}(\Omega)$, bounded from $L^{p_1}(\Omega)$ to $L^{q_1}(\Omega)$, and satisfying $|T(f)(x)|\leq|T(|f|)(x)|$. Set $d=1$, $\rho=|\cdot|$, and define an associated operator on $L^{p_0}_\mathcal{K}(\Omega,|\cdot|)+L^{p_1}_\mathcal{K}(\Omega,|\cdot|)$ as
$$T(F)(x)=[-|T_0(|F|)(x)|,|T_0(|F|)(x)|],\ \ for\ x\in\Omega.$$
It's easy to obtain that $T$ is a sublinear operator bounded from $L^{p_0}_\mathcal{K}(\Omega,|\cdot|)$ to $L^{q_0}_\mathcal{K}(\Omega,|\cdot|)$, and bounded from $L^{p_1}_\mathcal{K}(\Omega,|\cdot|)$ to $L^{q_1}_\mathcal{K}(\Omega,|\cdot|)$. Therefore, by Theorem 3.1, $T$ is bounded from $L^{p}_\mathcal{K}(\Omega,|\cdot|)$ to $L^{q}_\mathcal{K}(\Omega,|\cdot|)$.\\
For $g\in L^p(\Omega)$, define $G(x)=[-|g(x)|,|g(x)|]$ for $x\in\Omega$, we have
$$\lVert T_0g\rVert_q\leq\lVert T_0(|g|)\rVert_q=\lVert T_0(|G|)\rVert_q=\lVert TG\rVert_{L^q_{\mathcal{K}}(\Omega,|\cdot|)}\lesssim\lVert G\rVert_{L^p_{\mathcal{K}}(\Omega,|\cdot|)}=\lVert g\rVert_p,$$
which implies that $T_0$ is bounded from $L^p(\Omega)$ to $L^q(\Omega)$.
\end{remark}

Next, we define the fractional averaging operator and fractional maximal operator of locally integrably bounded functions, and use Theorem 3.1 to obtain their boundedness.

\begin{definition}
Let $F:\mathbb{R}^n\rightarrow\mathcal{K}_{bcs}(\mathbb{R}^d)$ that is locally integrably bounded, for a fixed cube $Q$ and $\alpha\in(0,1)$, define the fractional averaging operator $A_{Q,\alpha}$ by
$$A_{Q,\alpha}F(x)=\frac{1}{|Q|^{1-\alpha}}\int_{Q}F(y)dy\cdot\chi_Q(x).$$
\end{definition}

\begin{lemma}
Given any cube $Q$, the fractional averaging operator is linear: if $F,G:\mathbb{R}^n\rightarrow\mathcal{K}_{bcs}(\mathbb{R}^d)$ are locally integrably bounded mappings, $\lambda\in\mathbb{R}$, then for all $x\in\mathbb{R}^n$,
$$A_{Q,\alpha}(F+G)(x)=A_{Q,\alpha}F(x)+A_{Q,\alpha}G(x),\ \ A_{Q,\alpha}(\lambda F)(x)=\lambda A_{Q,\alpha}F(x).$$
\end{lemma}

\begin{proof}
By Lemma 2.1, we have
\begin{align*}
A_{Q,\alpha}(F+G)(x)&=\frac{1}{|Q|^{1-\alpha}}\int_{Q}(F(y)+G(y))dy\cdot\chi_Q(x)\\
&=\frac{1}{|Q|^{1-\alpha}}\int_{Q}F(y)dy\cdot\chi_Q(x)+\frac{1}{|Q|^{1-\alpha}}\int_{Q}G(y)dy\cdot\chi_Q(x)\\
&=A_{Q,\alpha}F(x)+A_{Q,\alpha}G(x),
\end{align*}
and
\begin{align*}
A_{Q,\alpha}(\lambda F)(x)&=\frac{1}{|Q|^{1-\alpha}}\int_{Q}\lambda F(y)dy\cdot\chi_Q(x)=\frac{\lambda}{|Q|^{1-\alpha}}\int_{Q}F(y)dy\cdot\chi_Q(x)=\lambda A_{Q,\alpha}F(x),
\end{align*}
which show the linearity.
\end{proof}

\begin{theorem}
Given a cube $Q$, for $0<\alpha<1$, $1\leq p<q\leq\infty$ such that $1/p-1/q=\alpha$, the fractional averaging operator $A_{Q,\alpha}:L^p_\mathcal{K}(\mathbb{R}^n,|\cdot|)\rightarrow L^q_\mathcal{K}(\mathbb{R}^n,|\cdot|)$ is bounded.
\end{theorem}

\begin{proof}
We begin with $p=1,\ q=1/(1-\alpha)$. By Lemma 2.2,
\begin{align*}
\lVert A_{Q,\alpha}F\rVert_{L_{\mathcal{K}}^{\frac{1}{1-\alpha}}(\mathbb{R}^n,|\cdot|)} &=\left(\int_{\mathbb{R}^n}\left|\frac{1}{m_n(Q)^{1-\alpha}}\int_QF(y)dy\cdot\chi_Q(x)\right|^{\frac{1}{1-\alpha}}dx\right)^{1-\alpha}\\
&\leq\left(\frac{1}{m_n(Q)}\int_{Q}\left(\int_Q|F(y)|dy\right)^{\frac{1}{1-\alpha}}dx\right)^{1-\alpha}\leq\lVert F\rVert_{L_{\mathcal{K}}^{1}(\mathbb{R}^n,|\cdot|)}.
\end{align*}

Then, consider the case $p=1/\alpha$, $q=\infty$. By Lemma 2.2 and H\"{o}lder's inequality, for all $x\in\mathbb{R}^n$, we have
\begin{align*}
|A_{Q,\alpha}F(x)|&=\left|\frac{1}{m_n(Q)^{1-\alpha}}\int_QF(y)dy\cdot\chi_Q(x)\right|\leq\frac{1}{m_n(Q)^{1-\alpha}}\int_Q|F(y)|dy\\
&\leq\left(\int_Q|F(y)|^{\frac{1}{\alpha}}dy\right)^{\alpha}\leq\lVert F\rVert_{L_{\mathcal{K}}^{\frac{1}{\alpha}}(\mathbb{R}^n,|\cdot|)},
\end{align*}
which implies that
$$\lVert A_{Q,\alpha}F\rVert_{L_{\mathcal{K}}^{\infty}(\mathbb{R}^n,|\cdot|)}\leq\lVert F\rVert_{L_{\mathcal{K}}^{\frac{1}{\alpha}}(\mathbb{R}^n,|\cdot|)}.$$

Finally, by Theorem 3.1 with $T=A_{Q,\alpha}$, $\Omega=\mathbb{R}^n$, $\rho(\cdot)=|\cdot|$, $p_0=1/\alpha$, $p_1=1$, $q_0=\infty$ and $q_1=1/(1-\alpha)$, we finish the proof.
\end{proof}

\begin{definition}
Given a locally integrably bounded function $F:\mathbb{R}^n\rightarrow\mathcal{K}_{bcs}(\mathbb{R}^d)$ and $\alpha\in(0,1)$, define the fractional maximal operator $M_{\alpha}$ by
$$M_{\alpha}F(x)=\overline{\mathrm{conv}}\left(\bigcup_QA_{Q,\alpha}F(x)\right),$$
where the union is taken over all cubes $Q$ whose sides are parallel to the coordinate axes.
\end{definition}

\begin{lemma}
The fractional maximal operator is sublinear: if $F,G:\mathbb{R}^n\rightarrow\mathcal{K}_{bcs}(\mathbb{R}^d)$ are locally integrably bounded mappings, $\lambda\in\mathbb{R}$, then for all $x\in\mathbb{R}^n$,
$$M_{\alpha}(F+G)(x)\subset M_{\alpha}F(x)+M_{\alpha}G(x),\ \ M_{\alpha}(\lambda F)(x)=\lambda M_{\alpha}F(x).$$
\end{lemma}

\begin{proof}
By Lemma 3.1 and the linearity of the closure of convex hull,
\begin{align*}
M_{\alpha}(F+G)(x)&=\overline{\mathrm{conv}}\left(\bigcup_QA_{Q,\alpha}(F(x)+G(x))\right)\subset\overline{\mathrm{conv}}\left(\bigcup_QA_{Q,\alpha}F(x)+\bigcup_QA_{Q,\alpha}G(x)\right)\\
&=\overline{\mathrm{conv}}\left(\bigcup_QA_{Q,\alpha}F(x)\right)+\overline{\mathrm{conv}}\left(\bigcup_QA_{Q,\alpha}G(x)\right)=M_{\alpha}F(x)+M_{\alpha}G(x),
\end{align*}
and
\begin{align*}
M_{\alpha}(\lambda F)(x)&=\overline{\mathrm{conv}}\left(\bigcup_QA_{Q,\alpha}(\lambda F(x))\right)=\lambda\overline{\mathrm{conv}}\left(\bigcup_QA_{Q,\alpha}(F(x))\right)=\lambda M_{\alpha}(F)(x),
\end{align*}
which show the sublinearity.
\end{proof}

To prove the boundedness of fractional maximal operators, we need the following results.

\begin{definition}\cite{Cru2016}
For $\tau\in\{0,\pm1/3\}^n$, define the translated dyadic grid
$$\mathcal{D}^{\tau}=\{2^k([0,1)^n+m+(-1)^k\tau):k\in\mathbb{Z},m\in\mathbb{Z}^n\},$$
and define the generalized dyadic fractional maximal operator
$$M^{\tau}_{\alpha}F(x)=\overline{\mathrm{conv}}\left(\bigcup_{Q\in\mathcal{D}^{\tau}}A_{Q,\alpha}F(x)\right),$$
which generalizes the standard dyadic grid $\mathcal{D}=\mathcal{D}^{0,0,\cdots,0}$, and dyadic fractional maximal operator
$$M^{d}_{\alpha}F(x)=M^{0,0,\cdots,0}_{\alpha}F(x).$$
\end{definition}

\begin{definition}\cite{Mar2022}
Given a locally integrably bounded function $F:\mathbb{R}^n\rightarrow\mathcal{K}_{bcs}(\mathbb{R}^d)$ and $\alpha\in(0,1)$, define
$$\widehat{M}_{\alpha}^dF(x)=\overline{\bigcup_{Q\in\mathcal{D}}A_{Q,\alpha}F(x)}.$$
\end{definition}

Similar to \cite[Lemma 5.9]{Mar2022}, the following lemma can be proved, we omit the details here.
\begin{lemma}
Given a locally integrably bounded, convex-set valued function $F:\mathbb{R}^n\rightarrow\mathcal{K}_{bcs}(\mathbb{R}^d)$, we have
$$M_{\alpha}F(x)\subset C\sum_{\tau\in\{0,\pm1/3\}^n}M^{\tau}_{\alpha}F(x),$$
where the constant $C$ depends only on dimension $n$ and $\alpha$.
\end{lemma}


\begin{theorem}
For $0<\alpha<1$, $1<p<q\leq\infty$ such that $1/p-1/q=\alpha$, the fractional maximal operator $M_{\alpha}:L^p_\mathcal{K}(\mathbb{R}^n,|\cdot|)\rightarrow L^q_\mathcal{K}(\mathbb{R}^n,|\cdot|)$ is bounded. For $0<\alpha<1$, $p=1$, $q=1/(1-\alpha)$, $M_{\alpha}:L^1_\mathcal{K}(\mathbb{R}^n,|\cdot|)\rightarrow L^{1/(1-\alpha),\infty}_\mathcal{K}(\mathbb{R}^n,|\cdot|)$ is bounded.
\end{theorem}

\begin{proof}
First, by Lemma 3.3,
$$|M_{\alpha}F(x)|\leq C\sum_{\tau\in\{0,\pm1/3\}^n}|M_{\alpha}^{\tau}F(x)|,$$
and so it will suffice to prove the strong and weak type inequalities for $M_{\alpha}^{\tau}$. In fact, given that all of the dyadic grids $\mathcal{D}^{\tau}$ have the same properties as $\mathcal{D}$, it will suffice to prove them for the dyadic fractional maximal operator $M_{\alpha}^d$. Moreover, arguing as the authors did in the proof of\cite[Lemma 5.6]{Mar2022}, it will suffice to prove our estimates for the operator $\widehat{M}^d_{\alpha}$.

We begin with $p=1$, $q=1/(1-\alpha)$ by adapting the Calder\'{o}n-Zygmund decomposition to convex-set valued functions. Fix $\lambda>0$ and define
$$\Omega_{\lambda}^d=\{x\in\mathbb{R}^n:|\widehat{M}_{\alpha}^dF(x)|>\lambda\}.$$
If ${\Omega_{\lambda}^d}$ is empty, there is nothing to prove. Otherwise, given $x\in\Omega_{\lambda}^d$, there must exist a cube $Q\in\mathcal{D}$ such that $x\in Q$, and
$$\frac{1}{m_n(Q)^{1-\alpha}}\left|\int_{Q}F(y)dy\right|>\lambda.$$
We claim that among all the dyadic cubes containing $x$, there must be a largest one with this property. By Lemma 2.2,
$$\frac{1}{m_n(Q)^{1-\alpha}}\left|\int_{Q}F(y)dy\right|\leq\frac{1}{m_n(Q)^{1-\alpha}}\lVert f\rVert_{L^1_\mathcal{K}(\mathbb{R}^n,|\cdot|)}.$$
Since the right-hand side goes to $0$ as $m_n(Q)\rightarrow\infty$, we see that such a maximal cube must exist, denote this cube by $Q_x$. Since the set of dyadic cubes is countable, we can enumerate the set $\{Q_x:x\in\Omega_{\lambda}^d\}$ by $\{Q_j\}_{j\in\mathbb{N}}$. The cubes $\{Q_j\}$ must be disjoint, since if one were contained in the other, it would contradict the maximality. By the choice of these cubes, $\Omega_{\lambda}^d\subset\bigcup_jQ_j$, hence
\begin{align*}
m_n(\Omega_{\lambda}^d)&\leq\sum_{j}m_n(Q_j)\leq\frac{1}{\lambda^{\frac{1}{1-\alpha}}}\sum_{j}m_n(Q_j)\left|\frac{1}{m_n(Q_j)^{1-\alpha}}\int_{Q_j}F(y)dy\right|^{\frac{1}{1-\alpha}}\leq\left(\frac{1}{\lambda}\int_{\mathbb{R}^n}|F(x)|dx\right)^{\frac{1}{1-\alpha}}.
\end{align*}

Then we consider the case $p=1/\alpha$, $q=\infty$. Since $F\in L^{1/\alpha}_\mathcal{K}(\mathbb{R}^n,|\cdot|)$, by the $(1/\alpha,\infty)$ boundedness of $A_{Q,\alpha}$, for a.e. $x\in\Omega$, we have
$$|A_{Q,\alpha}F(x)|\leq C\lVert F\rVert_{L^{\frac{1}{\alpha}}_\mathcal{K}(\mathbb{R}^n,|\cdot|)},$$
which implies that
$$\lVert\widehat{M}^d_{\alpha}F\rVert_{L^{\infty}_\mathcal{K}(\mathbb{R}^n,|\cdot|)}\leq C\lVert F\rVert_{L^{\frac{1}{\alpha}}_\mathcal{K}(\mathbb{R}^n,|\cdot|)}.$$

Finally, by Theorem 3.1 with $T=\widehat{M}^d_{\alpha}$, $\Omega=\mathbb{R}^n$, $\rho(\cdot)=|\cdot|$, $p_0=1/\alpha$, $p_1=1$, $q_0=\infty$ and $q_1=1/(1-\alpha)$, we finish the proof.
\end{proof}

\section{Riesz-Thorin interpolation theorem and its application}

Complex method is often important in operator theory, see \cite{Bao2023, Chu2023}. In this section, we first prove Riesz-Thorin interpolation theorem, which needs some results of density, analytic functions and dual norms.

\begin{definition}\cite{Cas1977, Mar2022}
Given a norm $\rho_x$ on $\mathbb{R}^d$ and two compact sets $K_1,K_2\subset\mathbb{R}^d$, the corresponding Hausdorff distance function is defined by
$$d_{H,x}(K_1,K_2)=\max\{\sup_{v\in K_1}\inf_{w\in K_2}\rho_x(v-w),\sup_{v\in K_2}\inf_{w\in K_1}\rho_x(v-w)\}.$$
Hence, define the Hausdorff distance on $\mathcal{K}_{bcs}(\mathbb{R}^d)$ as
$$d_p(F,G)=\left(\int_{\Omega}d_{H,x}(F(x),G(x))^pd\mu(x)\right)^{\frac{1}{p}}.$$
\end{definition}

\begin{definition}\cite{Mar2022}
A measurable function $F:\Omega\rightarrow\mathcal{K}_{bcs}(\mathbb{R}^d)$ is called a simple function if $F$ takes only finitely many values $S_1,S_2,\cdots,S_m\in\mathcal{K}_{bcs}(\mathbb{R}^d)$, hence, it can be written in the form
$$F(x)=\sum_{k=1}^m\chi_{A_k}(x)S_k\ \ for\ x\in\Omega,$$
where $A_1,A_2,\cdots,A_m$ are disjoint measurable sets such that
$$\bigcup_{k=1}^mA_k=\Omega.$$
\end{definition}

\begin{lemma}\cite{Mar2022}
Every measurable function $F:\Omega\rightarrow\mathcal{K}_{bcs}(\mathbb{R}^d)$ is the pointwise limit of simple measurable functions with respect to the Hausdorff distance on $\mathcal{K}_{bcs}(\mathbb{R}^d).$
\end{lemma}

\begin{lemma}
For any $p>0$, $F\in L_{\mathcal{K}}^p(\Omega,\rho)$ and $\epsilon>0$, there exists a simple measurable function $G$ such as $d_p(F,G)<\epsilon$.
\end{lemma}

\begin{proof}
By the proof of\cite[Lemma 4.4]{Mar2022}, there exist simple measurable functions $\{F_n\}_{n\in\mathbb{N}}$ such that
$$F_1(x)\subset F_2(x)\subset\cdots,\ \ F(x)=\bigcup_{n\in\mathbb{N}}F_n(x)\ \ and\ \ \lim_{n\to\infty}d_{H,x}(F_n(x),F(x))=0$$
for any $x\in\Omega$. Therefore,
$$d_{H,x}(F_n(x),F(x))\leq d_{H,x}(\{0\},F(x)),$$
and
$$\left(\int_{\Omega}d_{H,x}(\{0\},F(x))^pd\mu(x)\right)^{\frac{1}{p}}=\lVert f
\rVert_{L_{\mathcal{K}}^p(\Omega,\rho)}<\infty.$$
By Lebesgue dominated covergence theorem,
\begin{align*}
\lim_{n\to\infty}d_p(F_n,F)&=\lim_{n\to\infty}\left(\int_{\Omega}d_{H,x}(F_n(x),F(x))^pd\mu(x)\right)^{\frac{1}{p}}\\
&=\left(\int_{\Omega}\lim_{n\to\infty}d_{H,x}(F_n(x),F(x))^pd\mu(x)\right)^{\frac{1}{p}}=0.
\end{align*}
Setting $G(x)=F_n(x)$ for sufficiently large $n$ finishes the proof.
\end{proof}

\begin{lemma}\cite{Rud1970}
Suppose $\mu$ is a complex measure on a measure space $X$, $\Omega$ is an open set in the plane, $\phi$ is a bounded function on $\Omega\times X$ such that $\phi(z,t)$ is a measurable function of $t$ for each $z\in\Omega$, and $\phi(z,t)$ is analytic in $\Omega$ for each $t\in X$. For $z\in\Omega$, define
$$f(z)=\int_X\phi(z,t)d\mu(t),$$
then $f$ is analytic in $\Omega$.
\end{lemma}

\begin{lemma}\cite{Gra2012}
Let $A$ be analytic on the open strip $S=\{z\in\mathbb{C}:0<\mathrm{Re}z<1\}$ and continuous on its closure such that
$$\sup_{z\in\bar{S}}e^{-a|\mathrm{Im}z|}\log|A(z)|\leq A<\infty$$
for some fixed $A$ and $a<\pi$. Then
$$|A(x)|\leq\exp\left\{\frac{\sin(\pi x)}{2}\int_{-\infty}^{\infty}\left[\frac{\log|A(iy)|}{\cosh(\pi y)-\cos(\pi x)}+\frac{\log|A(1+iy)|}{\cosh(\pi y)+\cos(\pi x)}\right]dy\right\}$$
whenever $0<x<1$.
\end{lemma}

\begin{definition}\cite{Mar2022}
Given a seminorm $p$, define $p^{\ast}:\mathbb{R}^d\rightarrow[0,\infty)$ by
\begin{equation}
p^{\ast}(v)=\sup_{w\in\mathbb{R}^d,p(w)\leq 1}|\langle v,w\rangle|.
\end{equation}
\end{definition}

\begin{definition}\cite{Mar2022}
If $\rho:\Omega\times\mathbb{R}^d\rightarrow[0,\infty)$ is a norm function, then $\rho^{\ast}:\Omega\times\mathbb{R}^d\rightarrow[0,\infty)$, defined by $\rho_x^{\ast}(v)=(\rho_x)^{\ast}(v)$, is a measurable norm function.
\end{definition}

\begin{definition}\cite{Bha2007, Mar2022}
Let $p_0$, $p_1$ be two norms, $\rho_0$, $\rho_1$ be two norm functions. For $0<t<1$, define their weighted geometric mean by
$$p_t(v)=p_0(v)^{1-t}p_1(v)^t,\ \ \rho_t(x,v)=\rho_0(x,v)^{1-t}\rho_1(x,v)^t\ \ for\ x\in\Omega,\ v\in\mathbb{R}^d.$$
Let $A$ and $B$ be two symmetric positive definite matrices, for $0<t<1$, define their weighted geometric mean by
$$A\#_tB=A^{\frac{1}{2}}(A^{-\frac{1}{2}}BA^{-\frac{1}{2}})^tA^{\frac{1}{2}}.$$
\end{definition}

The function $p_t$ may not be a norm. However, if we still use (4) to define $p_t^{\ast}$, it will be a norm.

\begin{lemma}\cite{Mar2022}
Let $p_0$, $p_1$ be two norms, then for $0<t<1$, $p_t^{\ast}$ is a norm.
Moreover, $p_t^{\ast\ast}$ is a norm, and for all $v\in\mathbb{R}^d$, $p_t^{\ast\ast}(v)\leq p_t(v)$.
\end{lemma}

\begin{theorem}
Let $1\leq p_0,p_1,q_0,q_1<\infty$, $\rho_0$, $\rho_1$ be norm functions on $\Omega$, $X$ be the set of all locally integrably bounded functions
$F:\Omega\rightarrow\mathcal{K}_{bcs}(\mathbb{R}^d)$, $T:X\rightarrow X$ be a linear, monotone operator. Suppose that
$$\lVert T(F)\rVert_{L^{q_0}_\mathcal{K}(\Omega,\rho_{0})}\leq M_0\lVert F\rVert_{L^{p_0}_\mathcal{K}(\Omega,\rho_0)},$$
$$\lVert T(F)\rVert_{L^{q_1}_\mathcal{K}(\Omega,\rho_{1})}\leq M_1\lVert F\rVert_{L^{q_0}_\mathcal{K}(\Omega,\rho_1)},$$
for all $F\in X$, and
$\rho_0(T(\chi_AS))$, $\rho_1(T(\chi_AS))$ are bounded for all sets $A\subset\mathbb{R}^n$, $S\subset\mathbb{R}^d$ that make $\rho_0(\chi_AS)$, $\rho_1(\chi_AS)$ bounded.
Then for any $0<\theta<1$ and simple function $F\in L^{p}_\mathcal{K}(\Omega,\rho)$, we have
$$\lVert T(F)\rVert_{L^{q}_\mathcal{K}(\Omega,\rho)}\leq M_{\theta}\lVert F\rVert_{L^{p}_\mathcal{K}(\Omega,\rho)},$$
where $M_{\theta}$ is a constant independent of $F$, and
\begin{equation}
\rho=\rho_{\theta}^{\ast\ast},\ \ \frac{1}{p}=\frac{1-\theta}{p_0}+\frac{\theta}{p_1},\ \ \frac{1}{q}=\frac{1-\theta}{q_0}+\frac{\theta}{q_1}.
\end{equation}
By density, $T$ is bounded from $L^{p}_\mathcal{K}(\Omega,\rho)$ to $L^{q}_\mathcal{K}(\Omega,\rho)$ for all $\rho$, $p$ and $q$ as in (5).
\end{theorem}

\begin{proof}
First, assume that $\lVert F\rVert_{L^{p}_\mathcal{K}(\Omega,\rho)}=1$. Let
$$F=\sum_{k=1}^m\chi_{A_k}S_k$$
be a simple function in $X$, where $S_k\in\mathcal{K}_{bcs}$ and $A_k$ are pairwise disjoint subsets of $\Omega$ with finite measure.
We need to control
$$\lVert T(F)\rVert_{L^{q}_\mathcal{K}(\Omega,\rho)}=\left(\int_{\Omega}\rho(T(F)(x))^qdx\right)^{\frac{1}{q}}=\sup_{g}\left|\int_{\Omega}\rho(T(F)(x))g(x)dx\right|,$$
where the supremum is taken over all simple functions $g\in L^{q'}(\Omega)$ with $L^{q'}(\Omega)$ norm not bigger than 1. Write
$$g=\sum_{j=1}^nb_je^{i\beta_j}\chi_{B_j},$$
where $b_j>0$, $\beta_j$ are real, and $B_j$ are pairwise disjoint subsets of $\Omega$ with finite measure.\\
Denote the open strip $S=\{z\in\mathbb{C}:0<\mathrm{Re}z<1\}$, and its closure $\bar{S}=\{z\in\mathbb{C}:0\leq\mathrm{Re}z\leq1\}$. For $z\in\bar{S}$, define
$$P(z)=\frac{p}{p_0}(1-z)+\frac{p}{p_1}z,\ \ Q(z)=\frac{q'}{q'_0}(1-z)+\frac{q'}{q'_1}z,$$
and
$$A(z)=\int_{\Omega}\rho_0(T(F)(x))^{(1-z)P(z)}\rho_1(T(F)(x))^{zP(z)}g_z(x)dx,$$
where
$$g_z=\sum_{j=1}^nb_j^{Q(z)}e^{i\beta_j}\chi_{B_j}.$$
We claim that $F(z)$ is analytic on $S$. By the linearity of $T$,
$$A(z)=\sum_{j=1}^nb_j^{Q(z)}e^{i\beta_j}\int_{\Omega}\rho_0\left(\sum_{k=1}^mT(S_k\chi_{A_k})(x)\right)^{(1-z)P(z)}\rho_1\left(\sum_{k=1}^mT(S_k\chi_{A_k})(x)\right)^{zP(z)}\chi_{B_j}(x)dx.$$
Consider the open set $S_{\lambda}:=\{z\in\mathbb{C}:0<\mathrm{Re}(z)<1,\ |\mathrm{Im}(z)|<\lambda\}$ for given $\lambda>0$, for $z\in S_{\lambda}$, by the sublinearity of $\rho_0$ and $\rho_1$,
\begin{align*}
|A(z)|&\leq\sum_{j=1}^nb_j^{|Q(z)|}\int_{\Omega}\left(\sum_{k=1}^m\rho_0(T(\chi_{A_k}S_k)(x))\right)^{|1-z||P(z)|}\\
&\ \ \ \ \times\left(\sum_{k=1}^m\rho_1(T(\chi_{A_k}S_k)(x))\right)^{|z||P(z)|}\chi_{B_j}(x)dx\\
&=:\sum_{j=1}^nb_j^{|Q(z)|}\int_{\Omega}\phi(z,x)dx.
\end{align*}
For all $x\in\Omega$ and $z\in S_{\lambda}$, $\rho_0(T(\chi_{A_k}S_k)(x))$, $\rho_1(T(\chi_{A_k}S_k)(x))$ are bounded by assumption, and $|Q(z)|$, $|1-z||P(z)|$, $|z||P(z)|$ are all bounded by a constant $C_{\lambda}$ independent of $z$, thus $\phi(z,x)$ is bounded on $S_{\lambda}\times\Omega$. By Lemma 4.3, $F(z)$ is analytic in $S_{\lambda}$, and then $F(z)$ is analytic in $S$ by the arbitrary of $\lambda$.\\
Besides, for all $z\in\bar{S}$,
\begin{align*}
\log|A(z)|&\leq\log\left(\sum_{j=1}^nb_j^{|Q(z)|}|B_j|C_1^{|1-z||P(z)|}C_2^{|z||P(z)|}\right)\\
&\lesssim\log(C_1^{|1-z||P(z)|}C_2^{|z||P(z)|})\lesssim(|1-z|+|z|)|P(z)|\\
&\lesssim|1-z|^2+|z|^2\lesssim 1+|\mathrm{Im}z|^2,
\end{align*}
thus we have
$$\sup_{z\in\bar{S}}e^{-|\mathrm{Im}z|}\log|A(z)|\lesssim\sup_{z\in\bar{S}}e^{-|\mathrm{Im}z|}(1+|\mathrm{Im}z|^2)=1,$$
therefore, by Lemma 4.4, for all $0<x<1$,
$$|A(x)|\leq\exp\left\{\frac{\sin(\pi x)}{2}\int_{-\infty}^{\infty}\left[\frac{\log|A(iy)|}{\cosh(\pi y)-\cos(\pi x)}+\frac{\log|A(1+iy)|}{\cosh(\pi y)+\cos(\pi x)}\right]dy\right\}.$$
Note that
\begin{align*}
&\int_{-\infty}^{\infty}\left[\frac{\log|A(iy)|}{\cosh(\pi y)-\cos(\pi x)}+\frac{\log|A(1+iy)|}{\cosh(\pi y)+\cos(\pi x)}\right]dy\\
&\ \ \ \ \lesssim\int_{-\infty}^{\infty}\left[\frac{1+|y|^2}{e^{|y|}-\cos(\pi x)}+\frac{1+|y|^2}{e^{|y|}+\cos(\pi x)}\right]dy<C(x)<\infty,
\end{align*}
we have $|A(x)|\leq M_x<\infty$. Set $x=\theta$, we obtain
$$|A(\theta)|=\int_{\Omega}\rho_0(T(F)(x))^{(1-\theta)P(\theta)}\rho_1(T(F)(x))^{\theta P(\theta)}g_{\theta}(x)dx\leq M_{\theta}.$$
Note that $P(\theta)=Q(\theta)=1$, $g_{\theta}(x)=g(x)$, by Lemma 4.5, we have
$$\int_{\Omega}\rho(T(F)(x))g(x)dx\leq M_{\theta}.$$
Therefore, $\lVert T(F)\rVert_{L^{q}_\mathcal{K}(\Omega,\rho)}\leq M_{\theta}$ for all simple functions $F\in L^{p}_\mathcal{K}(\Omega,\rho)$ satisfying $\lVert F\rVert_{L^{p}_\mathcal{K}(\Omega,\rho)}=1$. For an arbitrary simple function $F\in L^{p}_\mathcal{K}(\Omega,\rho)$, by the linearity of $T$ and $\rho$ (with number multiplication), we obtain the same conclusion.\\
Finally, for any $F\in L^{p}_\mathcal{K}(\Omega,\rho)$ and $\epsilon>0$, by Lemma 4.2, there exists simple measurable function $G\in X$ satisfying $G(x)\subset F(x)$ for all $x\in\Omega$, and
$$d_p(F,G)=\left(\int_{\Omega}d_{H,x}(F(x),G(x))^pd\mu(x)\right)^{\frac{1}{p}}<\epsilon.$$
Set $d(x)=d_{H,x}(F(x),G(x))$, then $d\in L^p(\Omega)$ and $\lVert d\rVert_p<\epsilon$. By the definition of Hausdorff distance, for all $x\in\Omega$,
$$F(x)\subset G(x)+d(x)\boldsymbol{\overline{B}}.$$
Then, by the linearity and monotonicity of $T$,
\begin{align*}
\lVert T(F)\rVert_{L^{q}_\mathcal{K}(\Omega,\rho)}&\leq\lVert T(G)\rVert_{L^{q}_\mathcal{K}(\Omega,\rho)}+\lVert T(d(\cdot)\boldsymbol{\overline{B}})\rVert_{L^{q}_\mathcal{K}(\Omega,\rho)}\leq M_{\theta}(\lVert G\rVert_{L^{p}_\mathcal{K}(\Omega,\rho)}+\lVert d(\cdot)\boldsymbol{\overline{B}}\rVert_{L^{p}_\mathcal{K}(\Omega,\rho)})\\
&=M_{\theta}(\lVert G\rVert_{L^{p}_\mathcal{K}(\Omega,\rho)}+\lVert d\rVert_p)<M_{\theta}(\lVert F\rVert_{L^{p}_\mathcal{K}(\Omega,\rho)}+\epsilon).
\end{align*}
By the arbitrary of $\epsilon$, we have $\lVert T(F)\rVert_{L^{q}_\mathcal{K}(\Omega,\rho)}\leq M_{\theta}\lVert F\rVert_{L^{p}_\mathcal{K}(\Omega,\rho)}$, which completes the proof.
\end{proof}

By the similar argument as Remark 3.1, Theorem 4.1 can also deduce Riesz-Theoin interpolation theorem on Lebesgue spaces with some additional conditions of operator $T$.

To prove the reverse factorization property, we need the $\mathcal{A}_{p,norm}$ condition of norm functions.

\begin{definition}\cite{Mar2022}
Fix $1\leq p<\infty$ and suppose $\rho(\cdot,v)\in L^p_{loc}$ for all $v\in\mathbb{R}^d$, define
$$\langle\rho\rangle_{p,Q}(v)=\left(\frac{1}{m_n(Q)}\int_Q\rho_x(v)^pdx\right)^{\frac{1}{p}}.$$
\end{definition}

\begin{definition}\cite{Mar2022}
Given a norm function $\rho:\mathbb{R}^n\times\mathbb{R}^d\rightarrow[0,\infty)$, then for $1\leq p<\infty$, $\rho\in\mathcal{A}_{p,norm}$ if for every cube $Q$ and $v\in\mathbb{R}^d$,
$$\langle\rho^{\ast}\rangle_{p',Q}(v)\lesssim\langle\rho\rangle_{p,Q}^{\ast}(v).$$
\end{definition}

\begin{definition}\cite{Mar2022}
A matrix mapping is called measurable if each of its components is a measurable function. Let $A:\Omega\rightarrow\mathcal{M}_d$ be a measurable matrix mapping, define an seminorm function $\rho_A$ by
$$\rho_A(x,v)=|A(x)v|,\ \ x\in\Omega,\ v\in\mathbb{R}^d.$$
\end{definition}

For each $x$, $\rho_A(x,\cdot)$ is a seminorm, and since $A$ is measurable, the map $x\rightarrow\rho_A(x,v)$ is measurable for all $v$. Reversely, for a given norm function, the similar result exists.

\begin{lemma}\cite{Mar2022}
Let $\rho$ be a norm function, then there exists an associated measurable matrix mapping $W:\Omega\rightarrow\mathcal{S}_d$ such that for a.e. $x\in\Omega$ and every $v\in\mathbb{R}^d$, $W(x)$ is positive definite, and
$$\rho_W(x,v)\leq\rho(x,v)\leq\sqrt{d}\rho_W(x,v).$$
\end{lemma}

The following lemma shows the connection between norm functions and their associated matrix weights.

\begin{lemma}\cite{Mar2022}
Given a norm function $\rho:\mathbb{R}^n\times\mathbb{R}^d\rightarrow[0,\infty)$ with associated matrix mapping $W$ and $1<p<\infty$, $\rho\in\mathcal{A}_{p,norm}$ if and only if $W\in\mathcal{A}_{p,matrix}$, that is,
$$[W]_{\mathcal{A}_{p,matrix}}:=\sup_{Q}\left(\frac{1}{m_n(Q)}\int_Q\left(\frac{1}{m_n(Q)}\int_Q\lVert W(x)W^{-1}(y)\rVert_{\mathrm{op}}^{p'}dy\right)^{\frac{p}{p'}}dx\right)^{\frac{1}{p}}<\infty.$$
When $p=1$, $\rho\in\mathcal{A}_{1,norm}$ if and only if $W\in\mathcal{A}_{1,matrix}$, that is,
$$[W]_{\mathcal{A}_{1,matrix}}:=\sup_Q\mathop{\mathrm{ess}\sup}\limits_{x\in Q}\frac{1}{m_n(Q)}\int_Q\lVert W^{-1}(x)W(y)\rVert_{\mathrm{op}}dy<\infty.$$
\end{lemma}

We need the following characterization of norm weights.

\begin{definition}\cite{Mar2022}
Let $F:\mathbb{R}^n\rightarrow\mathcal{K}_{bcs}(\mathbb{R}^d)$ that is locally integrably bounded, for a fixed cube $Q$, we define the averaging operator $A_Q$ by
$$A_QF(x)=\frac{1}{|Q|}\int_{Q}F(y)dy\cdot\chi_Q(x).$$
\end{definition}

\begin{lemma}\cite{Mar2022}
Given $1\leq p<\infty$ and a norm function $\rho$, the following propositions are equivalent:

(i) $\rho\in\mathcal{A}_{p,norm}$;

(ii) Given any cube $Q$, $A_Q:L^p_{\mathcal{K}}(\mathbb{R}^n,\rho)\rightarrow L^p_{\mathcal{K}}(\mathbb{R}^n,\rho)$ is bounded.
\end{lemma}

We also need the following results about the weighted geometric mean.

\begin{lemma}\cite{Mar2022}
Suppose that $A,B\in\mathcal{S}_d$, and the norms $p_0$ and $p_1$ are given by
$$p_0(v)=|A^{\frac{1}{2}}v|\ \ and\ \ p_1(v)=|B^{\frac{1}{2}}v|\ \ for\ v\in\mathbb{R}^d,$$
then the double dual of the weighted geometric mean $p_t$ satisfies
$$p_t^{\ast\ast}(v)\approx|(A\#_tB)^{\frac{1}{2}}v|\ \ for\ v\in\mathbb{R}^d.$$
\end{lemma}

Now, we prove the reverse factorization property.
\begin{theorem}
Given $1\leq p_0,p_1<\infty$, suppose that matrix mappings $W_0\in\mathcal{A}_{p_0,matrix}$, $W_1\in\mathcal{A}_{p_1,matrix}$, and $\sup_{|v|=1}|W_0(x)v|$, $\sup_{|v|=1}|W_1(x)v|$ are bounded for $x\in\mathbb{R}^n$. Then for any $0<t<1$, $\overline{W}=(W_0^2\#_tW_1^2)^{1/2}\in\mathcal{A}_{p,matrix}$, where
$$\frac{1}{p}=\frac{1-t}{p_0}+\frac{1}{p_1}.$$
\end{theorem}

\begin{proof}
Define the following norm functions:
$$\rho_0(x,v)=|W_0(x)v|,\ \ \rho_1(x,v)=|W_1(x)v|\ \ and\ \ \rho(x,v)=|\overline{W}(x)v|\ \ for\ x\in\Omega,\ v\in\mathbb{R}^d.$$
By the assumption of $W_0$ and $W_1$, Lemma 4.7 and Lemma 4.8, for any cube $Q$,
$$A_Q:L^{p_0}_{\mathcal{K}}(\mathbb{R}^n,\rho_0)\rightarrow L^{p_0}_{\mathcal{K}}(\mathbb{R}^n,\rho_0),\ \ A_Q:L^{p_1}_{\mathcal{K}}(\mathbb{R}^n,\rho_1)\rightarrow L^{p_1}_{\mathcal{K}}(\mathbb{R}^n,\rho_1).$$
Therefore, by Theorem 4.1,
$$A_Q:L^{p}_{\mathcal{K}}(\mathbb{R}^n,\rho_t^{\ast\ast})\rightarrow L^{p}_{\mathcal{K}}(\mathbb{R}^n,\rho_t^{\ast\ast}).$$
Using Lemma 4.8 again, we have $\rho_t^{\ast\ast}\in\mathcal{A}_p$. By Lemma 4.9, $\rho_t^{\ast\ast}(x,v)\approx\rho(x,v)$, which implies that $\rho\in\mathcal{A}_{p,norm}$. Finally by Lemma 4.7, we complete the proof.
\end{proof}

\bigskip \medskip\noindent
\textbf{\bf Acknowledgments}\\
The authors thank the referees for their careful reading and helpful comments which indeed improved the presentation of this article.\\
\textbf{\bf Funding information}\\
The research was supported by Natural Science Foundation of China (Grant No. 12061069).\\
\textbf{\bf Authors contributions}\\
All authors have accepted responsibility for the entire content of this manuscript and approved its submission.\\
\textbf{\bf Conflict of interest}\\
Authors state no conflict of interest.
\bigskip \medskip

\noindent Yuxun Zhang and Jiang Zhou\\
\medskip
\noindent
College of Mathematics and System Sciences, Xinjiang University, Urumqi 830046\\
\smallskip
\noindent{E-mail }:
\texttt{zhangyuxun64@163.com} (Yuxun Zhang);
\texttt{zhoujiang@xju.edu.cn} (Jiang Zhou)\\

\bibliographystyle{unsrt}
\bibliography{Reference.bib}
\end{document}